\documentclass[12pt, reqno]{amsart}
\usepackage{latexsym}
\usepackage{mathtools}
\usepackage{amsmath}
\usepackage{amsthm}
\usepackage{mathrsfs}
\usepackage{lmodern}
\usepackage[T1]{fontenc}
\usepackage{textcomp}
\usepackage[utf8]{inputenc}
\usepackage{color}
\usepackage{amssymb}
\usepackage{enumerate}
\usepackage[abbrev]{amsrefs}
\usepackage{indentfirst}
\usepackage{graphicx}
\usepackage{bmpsize}
\usepackage{hyperref}
\usepackage{enumitem}
\usepackage{url}
\usepackage{autobreak}
\usepackage[justification=justified]{caption}
\usepackage{bm}
\usepackage{here}
\usepackage{cmap}
\input glyphtounicode.tex        
\usepackage[labelformat=empty,subrefformat=parens]{subcaption}
\usepackage[pagewise]{lineno}
\newcommand*\patchAmsMathEnvironmentForLineno[1]{
  \expandafter\let\csname old#1\expandafter\endcsname\csname #1\endcsname
  \expandafter\let\csname oldend#1\expandafter\endcsname\csname end#1\endcsname
  \renewenvironment{#1}
     {\linenomath\csname old#1\endcsname}
     {\csname oldend#1\endcsname\endlinenomath}}
\newcommand*\patchBothAmsMathEnvironmentsForLineno[1]{
  \patchAmsMathEnvironmentForLineno{#1}
  \patchAmsMathEnvironmentForLineno{#1*}}
\AtBeginDocument{
\patchBothAmsMathEnvironmentsForLineno{equation}
\patchBothAmsMathEnvironmentsForLineno{align}
\patchBothAmsMathEnvironmentsForLineno{flalign}
\patchBothAmsMathEnvironmentsForLineno{alignat}
\patchBothAmsMathEnvironmentsForLineno{gather}
\patchBothAmsMathEnvironmentsForLineno{multline}
}
\renewcommand{\thefootnote}{} %footnote counter

\theoremstyle{plain} %text of this environment is typesetted in italics
\newtheorem{theorem}{\indent\sc Theorem}[section]
\newtheorem{lemma}[theorem]{\indent\sc Lemma}
\newtheorem{corollary}[theorem]{\indent\sc Corollary}
\newtheorem{proposition}[theorem]{\indent\sc Proposition}

\theoremstyle{definition} %text of this environment is typesetted in roman letters
\newtheorem{definition}[theorem]{\indent\sc Definition}
\newtheorem{remark}[theorem]{\indent\sc Remark}

\makeatletter
  
  \@addtoreset{equation}{section}
\makeatother
\newcommand{\cT}{\mathcal{T}}

\newcommand{\cC}{\mathcal{C}}
\newcommand{\bR}{\mathbb{R}}

\newcommand{\bZ}{\mathbb{Z}}
\newcommand{\bQ}{\mathbb{Q}}

\newcommand{\Ca}{\mathfrak{C}}
\newcommand{\PPSL}{\mathrm{PPSL}_2(\bR)}
\newcommand{\relmiddle}[1]{\mathrel{}\middle#1\mathrel{}}
\allowdisplaybreaks[3] %改ページを数式に許す

\newcommand\cinput[2]{\lower#1pt\hbox{\input{#2}}}

\begin{document}

\title{Distorted and undistorted subgroups of the Lodha--Moore group}
\author{Yuya Kodama}
\date{}
\renewcommand{\thefootnote}{\arabic{footnote}}  %number
\setcounter{footnote}{0} %footnote counter
\begin{abstract}
  We show that the Baumslag--Solitar group $BS(1,2)$ is undistorted in the Lodha--Moore group $G_0$ using an explicit lower bound for the word length of $G_0$.
  We also show that Thompson's group $F$ is distorted in $G_0$.
\end{abstract}

\maketitle

\section{Introduction}
%%%%%%%%%%%%%%%%%%%%%%%%%%%%%%%%%%%%%%%%%
Thompson's groups $F$, $T$, and $V$ were defined by Richard Thompson in the 1960s. 
They admit concrete descriptions by pairs of binary trees, by homeomorphisms on the Cantor set, and so on, which make detailed studies of these groups possible.
These groups have several remarkable properties:
they are finitely presented, $T$ and $V$ are infinite simple groups, and the amenability of $F$ remains a major open problem.
To understand their mysterious properties, many groups closely related to Thompson's groups have been introduced and studied.
For background on Thompson's groups, see \cite{cannon1996introductory}, for example.

The Lodha--Moore group $G_0$ was introduced by Lodha and Moore in \cite{lodha2016nonamenable}.
It is finitely presented, torsion-free, contains no non-abelian free subgroups, and is of type $F_\infty$ \cite{lodha2016nonamenable,lodha2020nonamenable}.
The group resembles Thompson's group $F$ in many ways, as it can be obtained from $F$ by adding one generator, but it is known to be non-amenable.
Therefore, it is natural to further investigate the relationship between $G_0$ and $F$.

However, little is known about the geometric properties of $G_{0}$.
Some geometric results are known \cite{lodha2020nonamenable, zaremsky2016hnn, lucy2024divergence, kodama2025lodha}, but compared with $F$, our current understanding is still quite limited.
For this reason, in this paper we study distortion as one of the fundamental geometric properties of $G_0$.

Given an inclusion $H \hookrightarrow G$, distortion compares the intrinsic word metric on $H$ with the metric induced from $G$.
More precisely, if the inclusion is a quasi-isometric embedding (resp.~is not a quasi-isometric embedding), then the group $H$ is called \textit{undistorted} (resp.~\textit{distorted}).
An undistorted embedding preserves the large-scale geometry of $H$, while distortion indicates that the ambient group compresses certain geometry of $H$. 

Thompson's groups, and groups similar to them, often admit several concrete realizations of elements, which makes it possible to estimate word length efficiently.
However, in the case of $G_0$, except for the bound obtained in \cite[Proposition 3.2]{lucy2024divergence}, almost no such bounds were previously known to the best of our knowledge. 
In this paper we construct more effective bounds for our purposes.
More precisely, using the realization of $G_0$ as a group of homeomorphisms on $\bR$, we define two lower bounds for the word length of $G_0$, and apply them to show that the Baumslag--Solitar group $BS(1,2)$ is undistorted (Theorem \ref{theorem_BS_distorted}).
In contrast, by examining a specific sequence of elements in $F$, we prove that Thompson's group $F$ is distorted in $G_0$ (Theorem \ref{theorem_F_distorted}).

If a Baumslag--Solitar group appears as a subgroup of a group, then its geometric properties are quite restricted.
For example, a group containing $BS(m,n)$ with $|m| \neq |n|$ cannot be semihyperbolic \cite[Proposition 7.17]{alonso1995semihyperbolic}.
From this point of view, the fact that $BS(1,2)$ is undistorted in $G_{0}$ may help clarify how $G_{0}$ differs from Thompson's group $F$.
Note that Thompson's group $V$, and therefore $F$, does not contain any Baumslag--Solitar subgroup with $|m| \neq |n|$ \cite{rover1999subgroups,zbMATH06659031,zbMATH06256444}.

Several related results on distortion have been obtained for Thompson's groups and for groups similar to them.
First, $F$ is undistorted in $T$, and $F, T$ are undistorted in $V$ since the word lengths of $F$ and $T$ are equivalent to the number of carets \cite{burillo2001metrics,zbMATH05518616,zbMATH02147011}.
Cyclic subgroups of $V$ are known to be undistorted \cite{zbMATH02147011,zbMATH06256444}.
Inside $F$ itself, Burillo \cite{zbMATH01271644} and Guba and Sapir \cite{guba1997diagram,zbMATH01406862} showed that subgroups isomorphic to $F^m \times \bZ^n$ can be quasi-isometrically embedded into $F$.
Cleary and Taback \cite{zbMATH01985412} investigated several quasi-isometrically embedded subgroups, and the subgroups $\bZ \wr \bZ$ and $F \wr \bZ$ are undistorted \cite{zbMATH05190386,zbMATH06407221}. Also, all finitely generated maximal subgroups are undistorted \cite{zbMATH08069398,zbMATH08098344}.
In a more general setting, it was shown that arbitrary finitely generated abelian subgroups of any diagram group are undistorted \cite{zbMATH01406862}.

In higher-dimensional Thompson's groups $2V$, Burillo and Cleary \cite{MR2734164} showed that $F$, $T$, and $V$ are at least exponentially distorted in $2V$.
In the Thompson--Stein groups $F(n_1,\dots,n_k)$, Wladis \cite{zbMATH05879448} showed that the natural inclusions $F(n_i)\hookrightarrow F(n_1,\dots,n_k)$ are at least exponentially distorted, while cyclic subgroups remain undistorted \cite{zbMATH05883467}.
These facts suggest that higher-dimensional Thompson's groups and the Stein--Thompson groups have geometries that differ from those of other Thompson's groups, and one may view the Lodha--Moore group as belonging to the same family.

This paper is organized as follows:
in Section \ref{section_preliminary}, we review the definitions of Thompson's group $F$ and the Lodha--Moore group $G_0$, as subgroups of the homeomorphism groups of $\bR$ and the Cantor set.
We also define an invariant which gives a lower bound for the word length.
In Section \ref{section_BS_undistorted}, we show that the logarithm of the invariant gives a lower bound for the word length of $G_0$, and apply it to show that $BS(1,2)$ is undistorted.
In Section \ref{section_F_distorted}, we construct a sequence of elements in $F < G_0$ which shows that $F$ is distorted in $G_0$.
%%%%%%%%%%%%%%%%%%%%%%%%%%%%%%%%%%%%%%%%%
\section{Preliminaries} \label{section_preliminary}
%%%%%%%%%%%%%%%%%%%%%%%%%%%%%%%%%%%%%%%%%
In this paper, if $f$ and $g$ are homeomorphisms on a topological space to itself, then $fg \coloneqq g \circ f$.
For an element $g$ of a finitely generated group $G$ with a finite generating set $S$, its word length with respect to $S$ is denoted by ${\|g\|}_S$.

\subsection{Thompson's group $F$}
We begin by recalling the definition of Thompson's group $F$.
For more details, see \cite{cannon1996introductory}, for example.
Let $\cT$ be the set consisting of pairs of rooted binary trees with the same number of leaves.
We define an equivalence relation on $\cT$ as follows: let $(T_1, T_2) \in \cT$ be a pair of binary trees with $n$ leaves.
We label the leaves of $T_1$ and $T_2$ with $1, \dots, n$ from left to right.
A binary tree consisting of the root and two leaves is called a \textit{caret}.
By attaching a caret to the $i$-th leaf of each $T_1$ and $T_2$ for some $i$, we obtain a pair of binary trees in $(T_1^\prime, T_2^\prime) \in \cT$.
We define $(T_1, T_2)$ to be equivalent to $(T_1^\prime, T_2^\prime)$, and let $\sim$ be the equivalence relation generated by these expansions.
\textit{Thompson's group $F$} is the set $\cT/{\sim}$ with the following operation:
let $a=(A_1, A_2)$ and $b=(B_1, B_2)$ be representatives of elements in $\cT/{\sim}$.
Then there exist $(A_1^\prime, A_2^\prime) \sim (A_1, A_2)$ and $(B_1^\prime, B_2^\prime) \sim (B_1, B_2)$ such that $A_2^\prime=B_1^\prime$.
By using these new pairs, we define $ab$ to be the equivalence class of $(A_1^\prime, B_2^\prime)$.
One verifies that this operation is independent of the choice of representative.
It is known that for any equivalence class of $F$, there exists a unique representative whose number of carets is minimal \cite{cannon1996introductory}.
For $x \in F$, we call such a representative $(T_1, T_2)$ \textit{reduced pair of $x$}, and write the number of carets of $T_1$ (or $T_2$) as $N(x)$.
It is also known that $F$ has the following group presentation:
\begin{align*}
  \langle x_0, x_1 \mid [x_0x_1^{-1}, x_0^{-1}x_1x_0], [x_0x_1^{-1}, x_0^{-2}x_1x_0^2]\rangle,
\end{align*}
where the generators $x_0$ and $x_1$ are illustrated in Figure \ref{figure_generators_F}.
\begin{figure}[tbp]
  \begin{center}
    \includegraphics[width=0.6\linewidth]{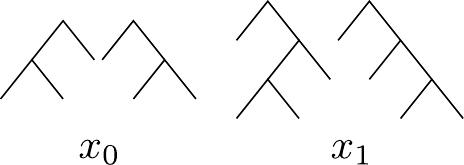}
  \end{center}
  \caption{Two generators of $F$. }
  \label{figure_generators_F}
\end{figure}
We denote the generating set $\{x_0, x_1\}$ of $F$ by $S_F$.

The following theorem is important for estimating the word length of $F$.
\begin{theorem}[{{\cite[Theorems 1, 3]{burillo2001metrics}}}] \label{theorem_Burillo_metric_Fp}
  There exists a constant $C \geq 1$ such that for any element $x \in F$,
  \begin{align*}
    \frac{1}{C} N(x) \leq \|x\|_{S_F} \leq CN(x)
  \end{align*}
  holds.
\end{theorem}

\subsection{The Lodha--Moore group $G_0$ as homeomorphisms on $\bR$}
In this section, we recall the definition of the Lodha--Moore group $G_0$.
The group admits two different but equivalent descriptions.
In this paper, we first introduce $G_0$ as a subgroup of $\PPSL$, the group of orientation-preserving piecewise projective homeomorphisms on $\bR \cup \{\infty\}$.
Precisely, for each element $f$ of $\PPSL$, there exist finitely many elements $s_1, \dots, s_n \in \bR$ such that $f\vert_{(-\infty, s_1]}, f\vert_{[s_1, s_2]}, \dots f\vert_{[s_{n-1}, s_n]}, f\vert_{[s_n, \infty)}$ are of the form $t \mapsto (a_1 t + b_1)/(c_1 t + d_1), t \mapsto (a_2 t + b_2)/(c_2 t + d_2), \dots, t \mapsto (a_n t + b_n)/(c_n t + d_n)$, respectively, where $a_i, b_i, c_i, d_i \in \bR$ with $a_id_i-b_i c_i=1$.
For $f \in \PPSL$, an element $t \in \bR$ is called a \textit{regular point} if there exists a neighborhood $U \subset \bR$ of $t$ such that $f\vert_U$ is of the form $(at+b)/(ct+d)$.
We denote $B(f) \coloneqq \bR \setminus \{\text{regular points of $f$}\}$ and call each element \textit{breakpoint}.
The connected components of $\bR \setminus B(f)$ are called \textit{linear fractional components}.
\begin{definition}[\cite{lodha2016nonamenable}]
  The \textit{Lodha--Moore group} $G_0$ is a subgroup of $\PPSL$ generated by the following maps:
  \begin{align*}
    a(t) & =t+1,
         &
    b(t) & = \left \{
    \begin{array}{cc}
      t              & \mbox{\rm{if} $t \leq 0$}                  \\
      \frac{t}{1-t}  & \mbox{\rm{if} $0 \leq t \leq \frac{1}{2}$} \\
      \frac{3t-1}{t} & \mbox{\rm{if} $\frac{1}{2} \leq t \leq 1$} \\
      t+1            & \mbox{\rm{if} $1 \leq t$},
    \end{array}
    \right.
         &
    c(t) & = \left \{
    \begin{array}{cc}
      \frac{2t}{1+t} & \mbox{\rm{if} $0 \leq t \leq 1$} \\
      t              & \mbox{\rm{otherwise}}.
    \end{array}
    \right.
  \end{align*}
\end{definition}
\begin{remark}
  The inverses of the above maps are the following:
  \begin{align*}
    a^{-1}(t) & =t-1,
              &
    b^{-1}(t) & = \left \{
    \begin{array}{cc}
      t             & \mbox{\rm{if} $t \leq 0$}        \\
      \frac{t}{t+1} & \mbox{\rm{if} $0 \leq t \leq 1$} \\
      \frac{1}{3-t} & \mbox{\rm{if} $1 \leq t \leq 2$} \\
      t-1           & \mbox{\rm{if} $2 \leq t$},
    \end{array}
    \right.
              &
    c^{-1}(t) & = \left \{
    \begin{array}{cc}
      \frac{t}{2-t} & \mbox{\rm{if} $0 \leq t \leq 1$} \\
      t             & \mbox{\rm{otherwise}}.
    \end{array}
    \right.
  \end{align*}
\end{remark}
On each linear fractional component, a map $t \mapsto (at+b)/(ct+d)$ is naturally identified with an element of the projective linear group $\mathrm{PGL}_2(\bR)$.

In the case of $G_0$, from the definition of the generators, on each linear fractional component, there exists a unique representative $A=\begin{pmatrix} a&b \\ c&d\end{pmatrix} \in \mathrm{GL}_2(\bZ)$ such that
\begin{enumerate}
  \item $\gcd(|a|,|b|,|c|,|d|)=1$;
  \item either $a>0$, or $a=0$ and $b>0$.
\end{enumerate}
By using this representative, we estimate the word length of elements of $G_0$ with respect to $\{a, b, c\}$.
\begin{definition} \label{definition_complexity}
  Let $f \in G_0$.
  \begin{enumerate}
    \item We define
          \begin{align*}
            D(f) \coloneqq \max \left\{q \in \bZ \relmiddle{|} \frac{p}{q} \in B(f), p \in \bZ, \gcd(|p|, q)=1, q>0\right\}.
          \end{align*}
          If $B(f)=\emptyset$, then we define $D(f) \coloneqq 1$;
    \item Let $I_1, \dots, I_n$ be all the linear fractional components of $f$. Then we define
          \begin{align*}
            M(f) \coloneqq \max_{1 \leq i \leq n} \left \{ \max\{|a_i|, |b_i|, |c_i|, |d_i|\} \relmiddle| \text{$\begin{pmatrix} a_i&b_i \\ c_i & d_i \end{pmatrix} \in \mathrm{GL}_2(\bZ)$ corresponds to $f\vert_{I_i}$}\right \}.
          \end{align*}
          In other words, $M(f)$ is the maximum absolute value of entries of matrices corresponding to $f\vert_{I_1}, \dots, f\vert_{I_n}$;
    \item We define $\cC(f) \coloneqq \max\{D(f), M(f)\}$.
  \end{enumerate}
\end{definition}
\begin{remark}
  By an unpublished result of Thurston, the subgroup $\langle a, b \rangle$ is known to be isomorphic to $F$.
\end{remark}
%%%%%%%%%%%%%%%%%%%%%%%%%%%%%%%%%%%%%%%%%
\subsection{The Lodha--Moore group as homeomorphisms on the Cantor set}
Let $\Ca$ be the Cantor set $\{0, 1\} \times \{0, 1\} \times \cdots$.
For a finite binary word $b_1$ and a finite (or infinite) binary word $b_2$, their concatenation is denoted by $b_1 b_2$.

In this section, we see $G_0$ as a subgroup of the homeomorphism group of $\Ca$.
To define generators of $G_0$, we first introduce the following homeomorphism.
\begin{definition}
  The map $y$ and its inverse $y^{-1}$ are defined recursively based on the following rule:
  \begin{align*}
     & y\colon \Ca \to \Ca        &  & y^{-1}\colon \Ca \to \Ca        \\
     & y(00\zeta)=0y(\zeta)       &  & y^{-1}(0\zeta)=00y^{-1}(\zeta)  \\
     & y(01\zeta)=10y^{-1}(\zeta) &  & y^{-1}(10\zeta)=01y(\zeta)      \\
     & y(1\zeta)=11y(\zeta),      &  & y^{-1}(11\zeta)=1y^{-1}(\zeta).
  \end{align*}
\end{definition}
For each finite binary word $s$, we also define the map $y_s$ by setting
\begin{align*}
  y_s(\xi) & =
  \left \{
  \begin{array}{cc}
    s y(\eta) & \xi=s\eta         \\
    \xi       & \mbox{otherwise}.
  \end{array}
  \right.
\end{align*}
Now we define two of the generators of $G_0$ as follows:
\begin{align*}
  x_0 & \colon \Ca \to\Ca;
  \begin{cases}
    00\eta \mapsto 0\eta  \\
    01\eta \mapsto 10\eta \\
    1\eta \mapsto 11 \eta,
  \end{cases}
      &
  x_1 & \colon \Ca \to\Ca;
  \begin{cases}
    0\eta \mapsto 0\eta     \\
    100\eta \mapsto 10\eta  \\
    101\eta \mapsto 110\eta \\
    11\eta \mapsto 111\eta.
  \end{cases}
\end{align*}
This convention is justified by the well-known identification between equivalence classes of elements in $F$ and homeomorphisms on $\Ca$.
See \cite{lodha2016nonamenable} for details.
In fact, the maps $a$ and $b$ can also be identified with $x_0$ and $x_1$.
\begin{proposition}[{\cite[Proposition 3.1]{lodha2016nonamenable}}] \label{proposition_G_0_piecewise_Cantor}
  The group $G_0$ is isomorphic to the group generated by $x_0, x_1$, and $y_{10}$.
  The isomorphism map is given by $a \mapsto x_0$, $b \mapsto x_1$, and $c \mapsto y_{10}$.
\end{proposition}
In particular, via this identification, we regard $F$ as a subgroup of $G_0$.

Here, we quickly recall the tree diagrams of $G_0$.
Because we only use this representation for $F < G_0$ in Section \ref{section_F_distorted}, see \cite{lodha2016nonamenable} for the precise definition.
We first note that for a given binary tree, each vertex is identified with a finite binary word in the following way:
we label all the left edges of carets with $0$ and the right one with $1$.
Then each vertex corresponds to a finite path on the tree from the root to the vertex, and the sequence of edge labels representing this path gives a finite binary word.

For a given finite binary word $s$, we take a binary tree $T_s$ such that $T_s$ is minimal and contains the path $s$ (from the root).
Then $y_s$ is represented by $(T_s, T_s^{\prime})$ where $T_s^\prime$ is the tree $T_s$ with the vertex $s$ labelled by a black dot.
Similar to $y_s$, the map $y_s^{-1}$ (resp.~$y_s^k$ with $|k|>1$) is represented by the pair consisting of two copies of $T_s$ where the vertex $s$ of the latter tree is labelled by a white dot (resp.~$k$).

In \cite[Remark 2.9]{zaremsky2016hnn}, it is claimed that $y_0^{-1}y_1$ and $x_0$ generate a copy of Baumslag--Solitar group $BS(2, 1)$.
Since this is not a subgroup of $G_0$ (but is a subgroup of the larger group ${}_yG_y=G$), we instead consider the following elements:
\begin{align*}
  g_1(t) & \coloneqq (bca^{-1}c^{-1}ab^{-1})(t)
  = \left \{
  \begin{array}{cc}
    t               & \mbox{\rm{if} $t \leq 0$}                  \\
    \frac{2t}{2t+1} & \mbox{\rm{if} $0 \leq t \leq \frac{1}{2}$} \\
    \frac{1}{3-2t}  & \mbox{\rm{if} $\frac{1}{2} \leq t \leq 1$} \\
    t               & \mbox{\rm{if} $1 \leq t$},
  \end{array}
  \right.
  \\
  g_2(t) & \coloneqq (bba^{-1}b^{-1}ab^{-1})(t)
  = \left \{
  \begin{array}{cc}
    t                 & \mbox{\rm{if} $t \leq 0$}                            \\
    \frac{t}{1-t}     & \mbox{\rm{if} $0 \leq t \leq \frac{1}{3}$}           \\
    \frac{4t-1}{5t-1} & \mbox{\rm{if} $\frac{1}{3} \leq t \leq \frac{1}{2}$} \\
    \frac{1}{2-t}     & \mbox{\rm{if} $\frac{1}{2} \leq t \leq 1$}           \\
    t                 & \mbox{\rm{if} $1 \leq t$}.
  \end{array}
  \right.
\end{align*}
By identifying $\Ca$ with $\{10\eta \mid \eta \in \Ca\}$, one directly verifies that these elements generate an isomorphic copy of $BS(2, 1)$.
See Figure \ref{figure_generators_BS} for the tree diagrams of $g_1$ and $g_2$.
\begin{figure}[tbp]
  \begin{center}
    \includegraphics[width=0.6\linewidth]{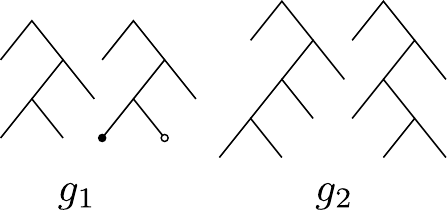}
  \end{center}
  \caption{Tree diagrams of $g_1$ and $g_2$. }
  \label{figure_generators_BS}
\end{figure}
Following the notation in \cite{zbMATH06481126}, we regard this group as a copy of $BS(1, 2)$ in this paper.
In other words, the map $\phi \colon BS(1, 2)=\langle x, t \mid txt^{-1}=x^2 \rangle \to \langle g_1, g_2 \rangle < G_0$ defined by $\phi(t)=g_1$ and $\phi(x)=g_2$ is an isomorphism.
By using this map, we see $BS(1, 2)$ as a subgroup of $G_0$.

Finally, we remark that every element in $BS(1, 2)$ admits a unique expression of the form
\begin{align*}
  t^{-m}x^{N}t^n
\end{align*}
with $m, n \geq 0$, and $N$ may be chosen to be an even number when either $m$ or $n$ is zero.
We refer to this as the \textit{normal form} of the element.
%%%%%%%%%%%%%%%%%%%%%%%%%%%%%%%%%%%%%%%%%$
\section{$BS(1, 2)$ is undistorted in $G_0$} \label{section_BS_undistorted}
\subsection{Lower bound of the word length of $G_0$}
We begin with some preliminary lemmas.
\begin{lemma} \label{lemma_M_upperbound}
  Let $f, g \in G_0$. Then we have $M(fg) \leq 2M(f)M(g)$.
\end{lemma}
\begin{proof}
  Let $A_1, \dots, A_n$ (resp.~$B_1, \dots, B_m$) denote the matrices associated to the linear fractional components of $f$ (resp.~$g$).
  Note that
  \begin{align*}
    \begin{pmatrix}
      a_{11} & a_{12} \\ a_{21}&a_{22}
    \end{pmatrix}
    \begin{pmatrix}
      b_{11} & b_{12} \\ b_{21}&b_{22}
    \end{pmatrix}
    =\begin{pmatrix}
       a_{11}b_{11}+a_{12}b_{21} & a_{11}b_{12}+a_{12}b_{22} \\ a_{21}b_{11}+a_{22}b_{21}&a_{21}b_{12}+a_{22}b_{22}
     \end{pmatrix}
  \end{align*}
  holds.
  Since we have
  \begin{align*}
    \max\{|a_{t1}b_{1k}+a_{t2}b_{2k}| \mid t, k \in \{1, 2\}\}
    \leq 2\max\{|a_{tk}| \mid t, k \in \{1, 2\}\}\max\{|b_{tk}| \mid t, k \in \{1, 2\}\},
  \end{align*}
  by taking maximum over $A_1, \dots, A_n$ and $B_1, \dots, B_m$, we have the desired result.
\end{proof}
\begin{lemma}\label{lemma_inverse_denominator}
  Let $f \in G_0$ and $I$ be a linear fractional component of $f$.
  Let $p^\prime/q^\prime \in f(I) \cap \bQ$ with $p^\prime, q^\prime \in \mathbb{Z}$, $q^\prime>0$, and $\gcd(|p^\prime|, q^\prime)=1$.
  Then for $p/q \coloneqq f^{-1}(p^\prime/q^\prime)$ with $p, q \in \mathbb{Z}$, $q>0$, and $\gcd(|p|, q)=1$, we have $q \leq 2 M(f) \max\{|p^\prime|, q^\prime\}$.
\end{lemma}
\begin{proof}
  If the restriction $f\vert_I$ is written as $f\vert_I(t)=(at+b)/(ct+d)$, then
  \begin{align*}
    \frac{p}{q}=\frac{d\frac{p^\prime}{q^\prime}-b}{-c\frac{p^\prime}{q^\prime}+a}=\frac{dp^\prime-bq^\prime}{-cp^\prime+aq^\prime}
  \end{align*}
  holds.
  Since $\gcd(|p|, q)=1$, we have
  \begin{align*}
    q \leq |-cp^\prime+aq^\prime| \leq |c||p^\prime|+|a|q^\prime \leq 2M(f) \max \{|p^\prime|, q^\prime\},
  \end{align*}
  as required.
\end{proof}
\begin{lemma}\label{lemma_composition_breakpoints}
  Let $f, g \in G_0$.
  Then we have
  \begin{align*}
    B(fg) \subset B(f) \cup f^{-1}(B(g)).
  \end{align*}
\end{lemma}
\begin{proof}
  Assume that $x \not\in B(f)$ and $f(x) \not \in B(g)$.
  It is sufficient to show that $x$ is not in $B(fg)$.

  Since $x \not \in B(f)$, there exists a neighborhood $U$ of $x$ on which $f\vert_U(t)=(at+b)/(ct+d)$.
  Similarly, since $f(x) \not \in B(g)$, there exists a neighborhood $V$ of $f(x)$ such that $g\vert_V(t)=(a^\prime t +b^\prime)/(c^\prime t +d^\prime)$.
  Hence $U \cap f^{-1}(V)$ is a neighborhood of $x$ contained in $U$, and its image under $f$ lies in $V$.
\end{proof}
%\begin{remark}
%  For $f, g \in G_0$, if $x \not \in B(f)$ and $f(x) \in B(g)$ hold, then $x$ is in $B(fg)$.
%  Indeed, if $x \not \in B(fg)$, then $f(x) \not \in B(g)$ either.
%\end{remark}
In the following, we write the standard generating set $\{a^{\pm 1}, b^{\pm 1}, c^{\pm 1}\}$ as $S_{G_0}$.
We compute how the functions defined in Definition \ref{definition_complexity} behave under the multiplication of generators.
\begin{lemma}\label{lemma_estimation_D(f)}
  Let $f \in G_0$.
  Then for any $s \in S_{G_0}$, we have
  \begin{align*}
    D(fs) \leq 4 \max\{D(f), M(f)\}=4\cC(f).
  \end{align*}
\end{lemma}
\begin{proof}
  Note that by Lemma \ref{lemma_composition_breakpoints}, we have
  \begin{align*}
    B(fs) \subset B(f) \cup f^{-1}(B(s)).
  \end{align*}
  Let $x \in B(fs)$ and write $x=p/q$ with $p, q \in \bZ$, $\gcd(|p|, q)=1$, and $q>0$.
  If $x \in B(f)$, clearly we have $q \leq D(f)$.

  Assume that $x \not \in B(f)$ and $f(x) \in B(s)$.
  Since $x \not \in B(f)$, there exists a neighborhood $U$ of $x$ such that $f\vert_U(t)=(at+b)/(ct+d)$ holds.
  From the definition of $M(f)$, we have $\max\{|a|, |b|, |c|, |d|\} \leq M(f)$.
  Also, since $f(x)=p^\prime / q^\prime \in B(s)$ and each generator has breakpoints with denominator and numerator at most two, we have $\max\{|p^\prime|, q^\prime\} \leq 2$.
  By applying Lemma \ref{lemma_inverse_denominator} to $x=f^{-1}(p^\prime/ q^\prime)$, we get
  \begin{align*}
    q
    \leq 2M(f) \max\{|p^\prime|, q^\prime\}
    \leq 4M(f),
  \end{align*}
  which implies the claim.
\end{proof}
\begin{lemma}\label{lemma_estimation_M(f)}
  Let $f \in G_0$. Then for any $s \in S_{G_0}$, we have
  \begin{align*}
    M(fs) \leq 6 M(f).
  \end{align*}
\end{lemma}
\begin{proof}
  Since $s \in S_{G_0}$, we have $M(s) \leq 3$.
  By Lemma \ref{lemma_M_upperbound}, we have
  \begin{align*}
    M(fs) \leq 2M(f)M(s) \leq 6M(f),
  \end{align*}
  as required.
\end{proof}
\begin{proposition}\label{proposition_estimation_C(f)}
  Let $f \in G_0$. Then for any $s \in S_{G_0}$, we have
  \begin{align*}
    \cC(fs) \leq 6\cC(f).
  \end{align*}
\end{proposition}
\begin{proof}
  By Lemmas \ref{lemma_estimation_D(f)} and \ref{lemma_estimation_M(f)}, we have
  \begin{align*}
    D(fs) & \leq 4\cC(f), & M(fs) & \leq 6M(f) \leq 6\cC(f).
  \end{align*}
  Therefore $\cC(fs)=\max\{D(fs), M(fs)\} \leq 6\cC(f)$.
\end{proof}
As a corollary, we obtain a lower bound for the word length. 
\begin{corollary} \label{corollary_wordlength_estimate_C}
  For every $f \in G_0$, we have
  \begin{align*}
    {\|f\|}_{S_{G_0}} \geq \frac{1}{2} \log \cC(f).
  \end{align*}
\end{corollary}
\begin{proof}
  Let ${\|f\|}_{S_{G_0}}=n$, and write $f=s_1 \cdots s_n$ as a geodesic word.
  By iteratively applying Proposition \ref{proposition_estimation_C(f)}, we have $\cC(f) \leq 6^n$.
  Hence
  \begin{align*}
    {\|f\|}_{S_{G_0}}=n \geq \log_6 \cC(f) \geq \frac{1}{2} \log \cC(f),
  \end{align*}
  as required.
\end{proof}
\subsection{Estimation of the word length of elements in $BS(1, 2)$ with respect to $G_0$}
We define $S_{BS(1, 2)}$ to be the generating set $\{g_1, g_2\}$ of $\langle g_1, g_2 \rangle$.
It suffices to prove that there exists $\lambda \geq 1$ and $k \geq 0$ such that for any $g \in \langle g_1, g_2 \rangle$, we have
\begin{align*}
  \frac{1}{\lambda}{\|g\|}_{S_{BS(1, 2)}}-k \leq {\|g\|}_{S_{G_0}}.
\end{align*}
To show this claim, we use the following result:
\begin{proposition}[{\cite[Proposition 2.1]{zbMATH06481126}}] \label{proposition_Burillo_BS}
  There exist constants $C_1, C_2, D_1, D_2>0$ such that for every element $x=t^{-m}a^Nt^n$ of $BS(1, 2)$ where $t^{-m}a^Nt^n$ is the normal form of $x$ and $N \neq 0$, we have
  \begin{align*}
    C_1 (m+n + \log |N|) -D_1 \leq \|x\| \leq C_2(m+n+ \log |N|)+D_2,
  \end{align*}
  where $\|x\|$ is the word length of $x$ with respect to $\{x, t\}$.
\end{proposition}
We begin by providing some basic lemmas regarding the iterates of $g_1$ and $g_2$.
\begin{lemma} \label{lemma_iterates_g1}
  Let $n$ be a nonnegative integer.
  Then for any $t \in [0, 1/2]$, we have
  \begin{align*}
    g_1^n(t)=\frac{2^n t}{2(2^n-1)t+1},
  \end{align*}
  and for any $t \in [1/2, 1]$, we have
  \begin{align*}
    g_1^{-n}(t)=\frac{(2^{n+1}-1)t+(1-2^n)}{(2^{n+1}-2)t+(2-2^n)}.
  \end{align*}
\end{lemma}
\begin{proof}
  These equalities can be easily verified by induction since $g_1([0, 1/2])=[0, 1/2]$ and $g_1([1/2, 1])=[1/2, 1]$ hold.
\end{proof}
\begin{lemma} \label{lemma_iterates_g2}
  Let $N$ be a positive integer. Then we have 
	\begin{align*}
    g_2^N\vert_{[0, \frac{1}{N+2}]}(t)&=\frac{t}{1-Nt}, & g_2^N\vert_{[\frac{1}{N+2}, \frac{1}{N+1}]}(t)=\frac{(N+3)t-1}{(N+4)t-1}. 
  \end{align*}
%  and $x_N \in B(g_2^N)$ holds.
\end{lemma}

\begin{proof}
  Since $g_2^N(1/(N+2))=g_2^{N-1}(1/(N+1))=\cdots =g_2(1/3)$ hold, the equality $g_2^N\vert_{I_N}(t)=t/(1-Nt)$ is obtained by composing $g_2(t)=t/(1-t)$ with itself $N$ times.

  We prove the second equality by induction.
  If $N=1$, it is clear.
  Assume that
  \begin{align*}
    g_2^{N-1}\vert_{[\frac{1}{N+1}, \frac{1}{N}]}(t)=\frac{(N+2)t-1}{(N+3)t-1}
  \end{align*}
  holds for $N \geq 2$.
  Since $[1/(N+2), 1/(N+1)] \subset [0, 1/3]$ holds, we have $g_2([1/(N+2), 1/(N+1)])=[1/(N+1), 1/N]$ and
  \begin{align*}
    g_2^N\vert_{[\frac{1}{N+2}, \frac{1}{N+1}]}(t)=g_2^{N-1}(\frac{t}{1-t})
    =\frac{(N+2)\frac{t}{1-t}-1}{(N+3)\frac{t}{1-t}-1}=\frac{(N+3)t-1}{(N+4)t-1},
  \end{align*}
  as required.
%  Finally, we can see that if
%  \begin{align*}
%    \frac{t}{1-Nt}=\frac{(N+3)t-1}{(N+4)t-1}
%  \end{align*}
%  holds, then we have $t=1/(N+2)$.
%  Therefore, $x_N$ is a breakpoint of $g_2^N$.
\end{proof}

We first consider the case of $N=0$.
\begin{lemma} \label{lemma_undistorted_N=0}
  Let $k \in \bZ$. Then
  \begin{align*}
    \frac{1}{4}{\|g_1^k\|}_{S_{BS(1, 2)}} \leq {\|g_1^k\|}_{S_{G_0}}
  \end{align*}
  holds.
\end{lemma}
\begin{proof}
  Since the word lengths of $g_1^k$ depend only on
  $|k|$, we may assume $k \geq 0$.
  By Lemma \ref{lemma_iterates_g1}, we have $M(f) \geq 2^k$.
  Hence by Corollary \ref{corollary_wordlength_estimate_C}, we have
  \begin{align*}
    {\|g_1^k\|}_{S_{G_0}} \geq \frac{1}{2} \log (\cC(g_1^k)) \geq \frac{1}{2} \log(2^k) \geq \frac{1}{4} \log_2(2^k)
    =\frac{1}{4}k
    \geq \frac{1}{4} {\|g_1^k\|}_{S_{BS(1, 2)}},
  \end{align*}
  as required.
\end{proof}
\begin{proposition} \label{proposition_BS_undistorted_Nneq0}
  Let $g \in \langle g_1, g_2 \rangle$ and $g_1^{-m} g_2^N g_1^n$ be the normal form of $g$ with $N \neq 0$.
  Then we have
  \begin{align*}
    {\|g\|}_{S_{G_0}}
    \geq \frac{1}{6C_2}{\|g\|}_{S_{BS(1, 2)}}-\frac{D_2}{6C_2}.
  \end{align*}
\end{proposition}
\begin{proof}
  %We may assume $N \geq 1$ since the word lengths of $g_1^{-m} g_2^N g_1^n$ and $(g_1^{-m} g_2^N g_1^n)^{-1}=g_1^{-n}g_2^{-N}g_1^m$ are the same.
  Suppose first that $N>0$.
  We claim that
  \begin{align}
    D(g) \geq N2^n+2 \label{eq_D(g)_lower_positive}
  \end{align}
  holds.
  We begin by observing that
  \begin{align*}
    y=\frac{1}{N2^n+2}
  \end{align*}
  is in $B(g)$.
  Since
  \begin{align*}
    g_2^N\left(g_1^n\left(\left[0, \frac{1}{N2^n+2}\right]\right)\right)
    =g_2^N\left(\left[0, \frac{1}{N+2}\right]\right)
    =\left[0, \frac{1}{2}\right]
  \end{align*}
  hold by Lemmas \ref{lemma_iterates_g1} and \ref{lemma_iterates_g2}, we have
  \begin{align*}
    (g_1^{-m}g_2^Ng_1^n)\vert_{[0, \frac{1}{N2^n+2}]}(t)
     & =g_1^{-m}\left(\frac{\frac{2^n t}{2(2^n-1)t+1}}{1-N\left(\frac{2^nt}{2(2^n-1)t+1}\right)}\right)         \\
     & =g_1^{-m}\left(\frac{2^n t}{\left((2-N)2^n-2\right)t+1}\right)                                           \\
     & = \frac{\frac{2^n t}{\left((2-N)2^n-2\right)t+1}}{-2(2^m-1)\frac{2^n t}{\left((2-N)2^n-2\right)t+1}+2^m} \\
     & = \frac{2^n t}{-(2^{m+1}-2^{n+1}+N2^{m+n})t+2^m}.
  \end{align*}
  On the other hand, since
  \begin{align*}
    g_2^N\left(g_1^n\left(\left[\frac{1}{N2^n+2}, \frac{1}{N2^n+2-2^n}\right]\right)\right)
    =g_2^N\left(\left[\frac{1}{N+2}, \frac{1}{N+1}\right]\right)
    =\left[\frac{1}{2}, \frac{2}{3}\right]
  \end{align*}
  hold by Lemmas \ref{lemma_iterates_g1} and \ref{lemma_iterates_g2} (and their proofs), we have
  \begin{align*}
    (g_1^{-m}g_2^Ng_1^n)\vert_{[\frac{1}{N2^n+2}, \frac{1}{N2^n+2-2^n}]}(t) & =
    g_1^{-m}\left(\frac{(N+3)\frac{2^n t}{2(2^n-1)t+1}-1}{(N+4)\frac{2^n t}{2(2^n-1)t+1}-1}\right)
    \\
                                                                            & =g_1^{-m}\left(\frac{(N2^n+2^n+2)t-1}{(N2^n+2^{n+1}+2)t-1}\right)
    \\
                                                                            & =\frac{(2^{m+1}-1)\frac{(N2^n+2^n+2)t-1}{(N2^n+2^{n+1}+2)t-1}+(1-2^m)}{(2^{m+1}-2)\frac{(N2^n+2^n+2)t-1}{(N2^n+2^{n+1}+2)t-1}+(2-2^m)} \\
                                                                            & = \frac{(N2^{m+n}+2^{m+1}+2^n)t-2^m}{(N2^{m+n}+2^{m+1}+2^{n+1})t-2^m}.
  \end{align*}
  A direct computation shows that
  \begin{align*}
    \frac{2^n t}{-(2^{m+1}-2^{n+1}+N2^{m+n})t+2^m}
    =
    \frac{(N2^{m+n}+2^{m+1}+2^n)t-2^m}{(N2^{m+n}+2^{m+1}+2^{n+1})t-2^m}
  \end{align*}
  holds only when $t=1/((N2^n+2))=y$.
  Hence $y$ is in $B(g)$ and inequality \eqref{eq_D(g)_lower_positive} holds.

  Next, when $N<0$, we claim that
  \begin{align}
    D(g) \geq |N|2^n+2 \label{eq_D(g)_lower_negative}
  \end{align}
  holds.
  Let $M=-N>0$.
  To show this, we prove that
  \begin{align*}
    y=\frac{M2^n+1}{M2^n+2}=\frac{-N2^n+1}{-N2^n+2}
  \end{align*}
  is in $B(g)$ by observing
  %Since
  %\begin{align*}
  %  g_2^N\left(g_1^n\left(\left[0, \frac{-N2^n+1}{-N2^n+2}\right]\right)\right)
  %  =g_2^N\left(\frac{M+1}{M+2}\right)
  %\end{align*}
  \begin{align*}
    g\vert_{\left[\frac{2^n(M-1)+1}{2^n(M-1)+2}, \frac{2^nM+1}{2^nM+2}\right]}, g\vert_{\left[\frac{2^nM+1}{2^nM+2}, 1\right]}.
  \end{align*}
  Since the two intervals are contained in $[1/2, 1]$,
  we have
  \begin{align*}
    g_1^n\vert_{\left[\frac{2^n(M-1)+1}{2^n(M-1)+2}, \frac{2^nM+1}{2^nM+2}\right]}(t)
    &=\frac{(2-2^n)t+(2^n-1)}{(2-2^{n+1})t+(2^{n+1}-1)}, \\
    g_1^n \vert_{\left[\frac{2^nM+1}{2^nM+2}, 1\right]}(t)
    &=\frac{(2-2^n)t+(2^n-1)}{(2-2^{n+1})t+(2^{n+1}-1)}.
  \end{align*}
  Note that we have
  \begin{align*}
    g_1^n\left(\frac{2^n(M-1)+1}{2^n(M-1)+2}\right) & =\frac{M}{M+1},
                                                    &
    g_1^n\left(\frac{2^nM+1}{2^nM+2}\right)
                                                    & =
    \frac{M+1}{M+2},
                                                    &
    g_1^n(1)                                        & =1.
  \end{align*}
  Similar to the proof of Lemma \ref{lemma_iterates_g2}, we have
  \begin{align*}
    g_2^{-M}\vert_{[\frac{M}{M+1}, \frac{M+1}{M+2}]}(t)=\frac{t-1}{(M+4)t-(M+3)}
  \end{align*}
  and
  \begin{align*}
    g_2^{-M}\vert_{[\frac{M+1}{M+2}, 1]}(t)=\frac{(M+1)t-M}{Mt+(1-M)}.
  \end{align*}
  Hence we have
  \begin{align*}
    g_2^{-M}(g_1^n\vert_{\left[\frac{2^n(M-1)+1}{2^n(M-1)+2}, \frac{2^nM+1}{2^nM+2}\right]}(t))
     & =\frac{\frac{(2-2^n)t+(2^n-1)}{(2-2^{n+1})t+(2^{n+1}-1)}-1}{(M+4)\frac{(2-2^n)t+(2^n-1)}{(2-2^{n+1})t+(2^{n+1}-1)}-(M+3)} \\
     & =\frac{2^nt-2^n}{(2^nM+2^{n+1}+2)t-(2^nM+2^{n+1}+1)}
  \end{align*}
  and
  \begin{align*}
    g_2^{-M}(g_1^n\vert_{\left[\frac{2^nM+1}{2^nM+2}, 1\right]}(t))
     & =\frac{(M+1)\frac{(2-2^n)t+(2^n-1)}{(2-2^{n+1})t+(2^{n+1}-1)}-M}{M\frac{(2-2^n)t+(2^n-1)}{(2-2^{n+1})t+(2^{n+1}-1)}+(1-M)} \\
     & =\frac{(2^nM-2^n+2)t-(2^nM-2^n+1)}{(2^nM-2^{n+1}+2)t-(2^nM-2^{n+1}+1)}. 
  \end{align*}
  We can see that
  \begin{align*}
    g_2^{-M}(\frac{M}{M+1})   & =\frac{1}{3},
                              &
    g_2^{-M}(\frac{M+1}{M+2}) & =\frac{1}{2},
                              &
    g_2^{-M}(1)               & =1
  \end{align*}
  hold.

  Consequently, we obtain
  \begin{align*}
    g\vert_{\left[\frac{2^n(M-1)+1}{2^n(M-1)+2}, \frac{2^nM+1}{2^nM+2}\right]}(t)
     & =\frac{\frac{2^nt-2^n}{(2^nM+2^{n+1}+2)t-(2^nM+2^{n+1}+1)}}{2(1-2^m)\frac{2^nt-2^n}{(2^nM+2^{n+1}+2)t-(2^nM+2^{n+1}+1)}+2^m} \\
     & =\frac{2^nt-2^n}{(2^{m+1}+2^{n+1}+2^{m+n}M)t-(2^{m}+2^{n+1}+2^{m+n}M)}
  \end{align*}
  and
  \begin{align*}
    g\vert_{\left[\frac{2^nM+1}{2^nM+2}, 1\right]}(t)
     & =\frac{(2^{m+1}-1)\frac{(2^nM-2^n+2)t-(2^nM-2^n+1)}{(2^nM-2^{n+1}+2)t-(2^nM-2^{n+1}+1)}+(1-2^m)}{(2^{m+1}-2)\frac{(2^nM-2^n+2)t-(2^nM-2^n+1)}{(2^nM-2^{n+1}+2)t-(2^nM-2^{n+1}+1)}+(2-2^m)} \\
     & = \frac{(2^{m+1}-2^n+2^{m+n}M)t-(2^m-2^n+2^{m+n}M)}{(2^{m+1}-2^{n+1}+2^{m+n}M)t-(2^m-2^{n+1}+2^{m+n}M)}.
  \end{align*}
  If
  \begin{align*}
    %g\vert_{\left[\frac{2^n(M-1)+1}{2^n(M-1)+2}, \frac{2^nM+1}{2^nM+2}\right]}(t)
    %=g\vert_{\left[\frac{2^nM+1}{2^nM+2}, 1\right]}(t)
    \frac{2^nt-2^n}{(2^{m+1}+2^{n+1}+2^{m+n}M)t-(2^{m}+2^{n+1}+2^{m+n}M)} \\
    =\frac{(2^{m+1}-2^n+2^{m+n}M)t-(2^m-2^n+2^{m+n}M)}{(2^{m+1}-2^{n+1}+2^{m+n}M)t-(2^m-2^{n+1}+2^{m+n}M)}
  \end{align*}
  holds, then $t=(2^nM+1)/(2^nM+2)$.
  Therefore $y$ is in $B(g)$. 
  Since $M 2^n+2 = |N|2^n+2$, inequality \eqref{eq_D(g)_lower_negative} holds.

  Now we estimate the word length of $g$ with respect to $S_{G_0}$.
  Since ${\|g\|}_{S_{G_0}}={\|g^{-1}\|}_{S_{G_0}}$ holds, if $N>0$, we have
  \begin{align*}
    {\|g\|}_{S_{G_0}}
    =\frac{1}{2}({\|g\|}_{S_{G_0}}+{\|g^{-1}\|}_{S_{G_0}})
     & \geq \frac{1}{4}(\log \cC(g)+\log \cC(g^{-1}))            \\
     & \geq \frac{1}{4}\left(\log(N2^n+2)+\log(|-N|2^m+2)\right) \\
     & \geq \frac{1}{4}(2\log(|N|)+n\log 2+m\log 2)              \\
     & \geq \frac{1}{6}(\log(|N|)+n+m)                           \\
     & \geq \frac{1}{6C_2}{\|g\|}_{S_{BS(1, 2)}}-\frac{D_2}{6C_2}
  \end{align*}
  by Corollary \ref{corollary_wordlength_estimate_C}, inequalities \eqref{eq_D(g)_lower_positive} and \eqref{eq_D(g)_lower_negative}, and Proposition \ref{proposition_Burillo_BS}.
  If $N<0$, since $\cC(g) \geq |N|2^n+2$ and $\cC(g^{-1}) \geq -N2^m+2$ hold, the same inequality follows.
\end{proof}
We are now ready to prove the main theorem.
\begin{theorem}\label{theorem_BS_distorted}
  The Baumslag--Solitar group $BS(1, 2)$ is an undistorted subgroup of the Lodha--Moore group $G_0$.
\end{theorem}
\begin{proof}
  We set
  \begin{align*}
    \frac{1}{\lambda}=\min\{\frac{1}{4}, \frac{1}{6C_2}\}, k=\frac{D_2}{6C_2}.
  \end{align*}
  Then for any $g \in \langle g_1, g_2 \rangle$, we have
  \begin{align*}
    {\|g\|}_{S_{G_0}}
    \geq \frac{1}{\lambda}{\|g\|}_{S_{BS(1, 2)}}-k
  \end{align*}
  by Lemma \ref{lemma_undistorted_N=0} and Proposition \ref{proposition_BS_undistorted_Nneq0}.
\end{proof}
\begin{remark}
  In \cite{kodama2023n}, the author defined a generalization of the group $G_0$, called $G_0(n)$.
  Since $G_0$ is a subgroup of $G_0(n)$, the group $BS(1, 2)$ is also a subgroup of $G_0(n)$. 
  However, to the best of our knowledge, it is not known whether $G_0(n)$ is a subgroup of $\PPSL$ or not.
  Hence, the proof of Theorem \ref{theorem_BS_distorted} cannot be directly extended to $G_0(n)$.
\end{remark}
%%%%%%%%%%%%%%%%%%%%%%%%%%%%%%%%%%%%%%%%%
\section{$F$ is distorted in $G_0$} \label{section_F_distorted}
In contrast to the previous section, here we regard elements of $F<G_0$ as pairs of binary trees. 
Recall that by using binary trees, their word lengths can be estimated by the number of carets (Theorem \ref{theorem_Burillo_metric_Fp}).
By analyzing a particular family of elements, we obtain the following: 
\begin{theorem} \label{theorem_F_distorted}
  The group $F$ is at least exponentially distorted in $G_0$.
\end{theorem}
\begin{proof}
  Let $a_n=y_{010}^n (x_0x_1^2x_0^{-1}x_1^{-1}x_0x_1^{-1}x_0^{-1})y_{0111}^{-n}y_{010}^{-n}y_{011}^n$.
  Figure \ref{figure_distorted_tree_F} shows the reduced pair of $x_0x_1^2x_0^{-1}x_1^{-1}x_0x_1^{-1}x_0^{-1}$.
  \begin{figure}[tbp]
    \begin{center}
      \includegraphics[width=0.35\linewidth]{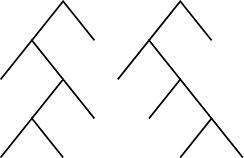}
    \end{center}
    \caption{The reduced pair of $x_0x_1^2x_0^{-1}x_1^{-1}x_0x_1^{-1}x_0^{-1}$. }
    \label{figure_distorted_tree_F}
  \end{figure}
  By induction, one checks that $a_n$ is in $F$ and $N(a_n)=2^{n}+3$.
  Figure \ref{figure_an} depicts the reduced pair of $a_n$.
  \begin{figure}[tbp]
    \begin{center}
      \includegraphics[width=0.8\linewidth]{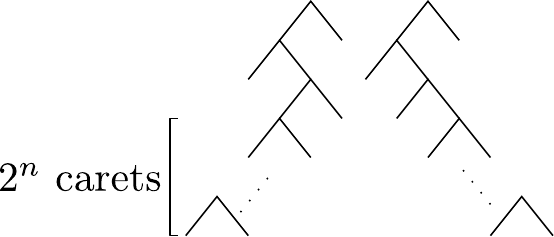}
    \end{center}
    \caption{The reduced pair of $a_n$. }
    \label{figure_an}
  \end{figure}
  Using the relations
  \begin{align*}
    y_{010}=(x_0x_1)y_{10}(x_0x_1)^{-1}, y_{0111}=(x_0x_1x_0^{-1}x_1x_0^{-1})y_{10}(x_0x_1x_0^{-1}x_1x_0^{-1})^{-1}
  \end{align*}
  and
  \begin{align*}
    y_{011}=(x_0x_1x_0^{-1})y_{10}(x_0x_1x_0^{-1})^{-1},
  \end{align*}
  we obtain a bound
  \begin{align*}
    {\|a_n\|}_{S_{G_0}} \leq 30+4n.
  \end{align*}

  Together with Theorem \ref{theorem_Burillo_metric_Fp}, this shows that $F$ is distorted in $G_0$.
\end{proof}
\begin{remark}
  A similar result also holds for $G_0(n)$.
  %Namely, the Brown--Thompson group $F(n)$ is distorted in $G_0(n)$.
  In fact, the same argument shows that the Brown--Thompson group $F(n)$ is distorted in $G_0(n)$.
\end{remark}
%%%%%%%%%%%%%%%%%%%%%%%%%%%%%%%%%%%%%%%%%
\section*{Acknowledgements}
The author was supported by JSPS KAKENHI Grant number 24K22836.
%%%%%%%%%%%%%%%%%%%%%%%%%%%%%%%%%%%%%%%%%
%%%%%%%%%%%%%%%%%%%%%%%%%%%%%%%%%%%%%%%%%
\bibliographystyle{plain}
\bibliography{references}
%%%%%%%%%%%%%%%%%%%%%%%%%%%%%%%%%%%%%%%%%
%%%%%%%%%%%%%%%%%%%%%%%%%%%%%%%%%%%%%%%%%
%%%%%%%%%%%%%%%%%%%%%%%%%%%%%%%%%%%%%%%%%

\bigskip
Yuya Kodama

\address{
  Graduate School of Science and Engineering, Kagoshima University,
  1-21-35 Korimoto, Kagoshima-city, Kagoshima, 890-0065, Japan
}

\textit{E-mail address}: \href{mailto:yuya@sci.kagoshima-u.ac.jp}{\texttt{yuya@sci.kagoshima-u.ac.jp}}
%%%%%%%%%%%%%%%%%%%%%%%%%%%%%%%%%%%%
\end{document}